\documentclass[preprint,12pt]{elsarticle}

\usepackage{amssymb}
\usepackage{amsmath, amsthm}
\usepackage{enumerate}
\usepackage{amssymb}
\usepackage{geometry}

\journal{...}

\begin{document}
	
\newtheorem{thm}{Theorem}[section]
\newtheorem{exam}[thm]{Example}
\newtheorem{cor}[thm]{Corollary}
\newtheorem{ques}[thm]{Question}
\newtheorem{lem}[thm]{Lemma}
\newtheorem{prop}[thm]{Proposition}
\newtheorem{ax}{Axiom}
\newtheorem{defn}[thm]{Definition}
\newtheorem{rem}{Remark}[section]
\newtheorem{rems}{Remarks}[section]
\newcommand{\thmref}[1]{Theorem~\ref{#1}}
\newcommand{\secref}[1]{\S\ref{#1}}
\newcommand{\lemref}[1]{Lemma~\ref{#1}}
\newcommand{\bysame}{\mbox{\rule{3em}{.4  pt}}\,}
\newcommand{\A}{\mathcal{A}}
\newcommand{\B}{\mathcal{B}}
\newcommand{\st}{\sigma}
\newcommand{\XcY}{{(X,Y)}}
\newcommand{\SX}{{S_X}}
\newcommand{\SY}{{S_Y}}
\newcommand{\SXY}{{S_{X,Y}}}
\newcommand{\SXgYy}{{S_{X|Y}(y)}}
\newcommand{\Cw}[1]{{\hat C_#1(X|Y)}}
\newcommand{\G}{{G(X|Y)}}
\newcommand{\PY}{{P_{\mathcal{Y}}}}
\newcommand{\X}{\mathcal{X}}
\newcommand{\wt}{\widetilde}
\newcommand{\wh}{\widehat}

\begin{frontmatter}



\title{Hahn-Banach theorem for operators on the lattice normed $f$-algebras}


\author{Abdullah Ayd{\i}n}

\address{Department of Mathematics, Mu\c{s} Alparslan University, Mu\c{s}, Turkey}

\begin{abstract}
Let $X$ and $E$ be $f$-algebras and $p:X \to E_+$ be a monotone vector norm. Then the triple $(X,p,E)$ is called a lattice-normed $f$-algebraic space. In this paper, we show a generalization of the extension of the Hahn-Banach theorem for operators on the lattice-normed $f$-algebras, in which the extension of one step of that is not similar to the other Hahn-Banach theorems. Also, we give some applications and results.
\end{abstract}


\end{frontmatter}


\section{Introductory Facts}\label{hb1}

The Hahn-Banach theorem has a lot of applications in different fields of analysis, which attracted the attention of several authors such as Vincent-Smith \cite{Vic} and Turan \cite{Tu}. In this present paper, we give an extension of the Hahn-Banach theorem on lattice normed $f$-algebras and some applications. The extension of one step in our theorem is not similar to the other Hahn-Banach theorems.

Vector lattices (i.e., Riesz spaces) are ordered vector spaces that have many applications in measure theory, operator theory, and applications in economics. We suppose that the reader to be familiar with the elementary theory of vector lattices, and we refer the reader for information on vector lattices \cite{ABPO,LZ,Za} as sources of unexplained terminology. Besides, all vector lattices are assumed to be real and Archimedean. A vector lattice $E$ is a {\em lattice-ordered algebra} (briefly, {\em $l$-algebra}) if $E$ is an associative algebra whose positive cone $E_+$ is closed under the algebra multiplication. A Riesz algebra $E$ is called \textit{$f$-algebra} if $E$ has additionally property that $x\wedge y=0$ implies $(x\cdot z)\wedge y=(z\cdot x)\wedge y=0$ for all $z\in E_+$. For an order complete vector lattice (i.e., Dedekind complete), the set $L_b(E)$ of all order bounded operators on $E$ and the set $C(X)$ of all real valued continuous function on a topological space $X$ are examples of lattice-ordered algebra. However, $L_b(E)$ is not $f$-algebra because it is Archimedean vector lattice but not commutative because every Archimedean $f$-algebra is commutative; see for example \cite[Theorem 140.10.]{Za}. Consider $Orth(E):=\{T\in L_b(E):x\perp y\ \text{implies}\ Tx\perp y\}$ the set of orthomorphisms on a vector lattice $E$. Then, the space $Orth(E)$ is not only vector lattice but also an $f$-algebra. On the other hand, a sublattice $A$ of an $f$-algebra $E$ is called $f$-subalgebre of $E$ whenever it is also an $f$-algebra under the multiplication operation in $E$. In this paper, we assume that if a positive element has inverse then the inverse also positive. We refer the reader for much more information on $f$-algebras \cite{ABPO,Ay1,Ay2,Hu,P,Za}. Also, for more details information on the following example, we refer the reader to \cite[p.13]{BGKKKM}.

\begin{exam}\label{example of orh}
Let $E$ be a vector lattice. An order bounded band preserving operator $T:D\to E$ on an order dense ideal $D\subseteq E$ is an extended orthomorphism. $Orth^\infty(E)$ denote the set of all extended orthomorphisms: denote by $\mathcal{M}$ the collection of all pairs $(D;\pi)$, where $D$ is order dense ideal in $E$ and $\pi\in Orth(D,E)$. Then the space $Orth^\infty(E)$ is an $f$-algebra. Moreover, $Orth(E)$ is an $f$-subalgebra of $Orth^\infty(E)$. On the other hand, $\mathcal{L}(E)$ stands for the order ideal generated by the identity operator $I_E$ in $Orth(E)$. Then $\mathcal{L}(E)$ is an $f$-subalgebra of $Orth(E)$.
\end{exam}

Recall that a net $(x_\alpha)_{\alpha\in A}$ in a vector lattice $X$ is called \textit{order convergent} (or shortly, \textit{$o$-convergent}) to $x\in X$, if there exists another net $(y_\beta)_{\beta\in B}$ satisfying $y_\beta \downarrow 0$ (i.e. $y_\beta \downarrow$ and $\inf(y_\beta)=0$), and for any $\beta\in B$ there exists $\alpha_\beta\in A$ such that $|x_\alpha-x|\leq y_\beta$ for all $\alpha\geq\alpha_\beta$. In this case, we write $x_\alpha\xrightarrow{o} x$. On the other hand, for a given positive element $u$ in a vector lattice $E$, a net $(x_\alpha)$ in $E$ is said to converge $u$-uniformly to the element $x\in E$ whenever, for every $\varepsilon>0$, there exists an index $\alpha_0$, such that $\lvert x_\alpha-x\rvert<\varepsilon u$ for every $\alpha\geq\alpha_0$. Moreover, $E$ is said to be $u$-uniformly complete if every $u$-uniform Cauchy net has an $u$-uniform limit; see \cite{LZ}.

Let $X$ be a vector space, $E$ be a vector lattice, and $p:X \to E_+$ be a vector norm (i.e. $p(x)=0\Leftrightarrow x=0$, 
$p(\lambda x)=|\lambda|p(x)$ for all $\lambda\in\mathbb{R}$, $x\in X$, and $p(x+y)\leq p(x)+p(y)$ for all $x,y\in X$), then the triple $(X,p,E)$ is called a {\em lattice-normed space}, abbreviated as $LNS$. A subset $Y$ of $X$ is called $p$-bounded whenever every net $(y_\alpha)$ in $Y$ with $p(y_\alpha-y)\xrightarrow{o} 0$ implies $y\in Y$. Let $(X,p,E)$ and $(Y,q,F)$ be two $LNS$s. Then an operator $T:X\to Y$ is called dominated operator if there is a positive operator $S:E\to F$ such that $q(T(x))\leq S(p(x))$ for all $x\in X$. In this case, $T$ is called a {\em dominated operator} and $S$ is called dominant of $T$. Take $maj(T)$ as the set of all dominants of the operator $T$. If there is a least element in $maj(T)$ then it is called the exact {\em dominant} of $T$ and denoted by $[T]$; see for much more details information see \cite{BGKKKM,Ku}. If $X$ is decomposable space and $F$ is order complete then exact dominant exists; see \cite[Theorem 4.1.2.]{Ku}. 

Consider an $LNS$ $(X,p,E)$. Assume $X$ and $E$ are $f$-algebras, and the vector norm $p$ is monotone (i.e. $|x|\leq |y|\Rightarrow p(x)\leq p(y)$) then the triple $(X,p,E)$ is said to be {\em lattice normed $f$-algebra} and abbreviated as $LNFA$.

\begin{defn}
Let $(X,p,E)$ be an $LNFA$ and $Y$ be an $f$-subalgebra of $X$. If $p(x\cdot y)=y\cdot p(x)$ holds for all $x\in X$ and $y\in Y$ then $p$ is said to be {\em $f$-subalgebraic-linear}. Also, we said that $(X,p,E)$ has {\em $f$-subalgebraic-linear property}.
\end{defn}

Recall that an element $x$ in Riesz algebra is called \textit{nilpotent} if $x^n=0$ for some $n\in \mathbb{N}$. Moreover, an algebra $E$ is called \textit{semiprime} if the only nilpotent element in $E$ is zero.
\begin{lem}\label{inequality semiprime}
Let $E$ be a semiprime $f$-algebra. Then $x\leq y$ and $x\leq z$ imply $x^2\leq y\cdot z$ for all $x,\ y,\ z\in E_+$.
\end{lem}

\begin{proof}
Suppose $x,\ y,\ z$ are positive elements in $E$ such that $x\leq y$ and $x\leq z$. It follows from \cite[Theorem 3.2.(ii)]{P} that $x^2\leq y\cdot z$.
\end{proof}

\begin{exam}
Let $E$ be a vector lattice such that $x^2=x$ for all $x\in E_+$ and $p:\mathcal{L}(E)\to Orth(E)$ be a map defined by $T\to p(T)=\lvert T\rvert$. Then one can see that $p$ is vector norm and $\big(\mathcal{L}(E),p, Orth(E)\big)$ is an $LNS$. Moreover, since $\mathcal{L}(E)$ and $Orth(E)$ are $f$-algebras and $\lvert\cdot\rvert$ is monotone, $\big(\mathcal{L}(E),p, Orth(E)\big)$ is an $LNFA$. Take arbitrary $T,\ S\in \mathcal{L}(E)$. Then there exists some positive scalars $\lambda_T$ and $\lambda_S$ such that $\lvert T\rvert\leq \lambda_T I$ and $\lvert S\rvert\leq \lambda_S I$ because $\mathcal{L}(E)$ is an order ideal generated by the identity operator $I_E$. So, by using \cite[Theorem 2.40.]{ABPO}, we have
$$
p(S(T))=\lvert S(T)\rvert=\lvert S\rvert\big(\lvert T\rvert\big)\leq \lambda_S I\big(\lvert T\rvert\big)=\lambda_S \lvert T\rvert
$$
and also
$$
p(S(T))=\lvert S(T)\rvert= \lvert S\rvert\big(\lvert T\rvert\big)\leq \lvert S\rvert\big(\lambda_T\lvert I\rvert\big)=\lambda_T \lvert S\rvert.
$$
So, it follows from Lemma \ref{inequality semiprime} and our assumption that $p(S(T))=\big[p(S(T))\big]^2\leq \lambda_S\lambda_T\lvert S\rvert \cdot \lvert T\rvert=\lambda_S\lambda_T\lvert S\rvert \cdot p(T)$ holds true because $Orth(E)$ is semiprime; see \cite[Theorem 142.5.]{Za}. Next, consider a new $LNFA$ $\big(\mathcal{L}(E)_+,q, Orth(E)\big)$, where $q(T)=\frac{1}{\lambda_T}p(T)$ for all $T\in \mathcal{L}(E)_+$. Then it follows from the above observation that the $LNFA$ space $\big(\mathcal{L}(E)_+,q,Orth(E)\big)$ has the $f$-subalgebraic-linear property.
\end{exam}

For the following example, we consider \cite[Theorem 2.62.]{ABPO}.
\begin{exam}
Let $E$ be an $f$-algebra. Then we define a map $p$ from $E$ to $Orth(E)$ by $u\to p(u)=p_u$ such that $p_u(x)=\lvert u\cdot x\rvert$ for each $x\in E$. So, by using \cite[Theorem 142.1.(ii)]{Za}, it is easy to see that $p$ is $(E,p,Orth(E))$ is an $LNFA$ with the $f$-subalgebraic-linear property. 
\end{exam}
In this article, unless otherwise, all lattice normed $f$-algebra are assumed to be with the $f$-subalgebraic-linear property.
\section{Main Results}
We begin the section with the following definition.
\begin{defn}
Let $(X,p,E)$ be an $LNS$. Then an operator $T:X\to E$ is said to be {\em $E$-dominated} if it is dominated by $p$ on $E$. It means that 
$$
\lvert T(x)\rvert\leq p(x)
$$
for all $x\in X$.
\end{defn}

It can be seen that every dominated operator on $LNS$s is $E$-dominated because dominant operators are positive.

\begin{lem}\label{f algebra subspace}
Let $X$ be an $f$-algebra and $Y$ be an $f$-subalgebra of $X$. Then, for any $w\in X_+$, the set $A=\{u+v\cdot w:u,v\in Y\}$ is also an $f$-subalgebra of $X$.
\end{lem}

\begin{proof}
Firstly, we show that $A$ is an sublattice of $X$. Take an arbitrary $u+v\cdot w\in A$. Then we have $\lvert u+v\cdot w\rvert=\lvert u\rvert+\lvert v\rvert\cdot \lvert w\rvert=\lvert u\rvert+\lvert v\rvert\cdot w\in A$ because of $\lvert u\rvert,\lvert v\rvert \in Y$. Then we get the desired result.

Next, we show that $A$ is an $f$-subalgebra of $X$. For any positive elements $y_1+u_1\cdot w,\ y_2+u_2\cdot w\in A_+$, we have 
$$
(y_1+u_1\cdot w)\cdot(y_2+u_2\cdot w)=y_1\cdot y_2+(y_1\cdot u_2+y_2\cdot u_1+u_1\cdot u_2\cdot w)w\in A_+
$$ 
because of $y_1\cdot y_2\in Y$, $y_1\cdot u_2+y_2\cdot u_1+u_1\cdot u_2\cdot w\in E$, $A\subseteq X$ and $X$ is $f$-algebra. Thus, $A$ is an $l$-algebra. On the other hand, assume $(y_1+u_1\cdot w)\wedge(y_2+u_2\cdot w)=0$ for arbitrary $y_1+u_1\cdot w,\ y_2+u_2\cdot w\in A$. Then we have $[(y+u\cdot w)\cdot(y_1+u_1\cdot w)]\wedge(y_2+u_2\cdot w)=0$ for all $y+u\cdot w\in A_+$ because $A_+\subseteq X_+$ and $X$ is $f$-algebra. Therefore, we obtain that $A$ is a $f$-subalgebra of $X$.
\end{proof}

\begin{prop}\label{f algebra order complete}
Let $X$ be an $f$-algebra and $Y$ be an $u$-uniformly complete $f$-subalgebra of $X$. Then, for any $w\in X_+$, the set $A=\{u+v\cdot w:y,z\in Y_+\}$ is also an $u$-uniformly complete $f$-subalgebra.
\end{prop}

\begin{proof}
Suppose $Y$ is $f$-subalgebra of $X$. Then, by applying Lemma \ref{f algebra subspace}, we see that $A$ is $f$-subalgebra of $X$. On the other hand, take an $u$-uniform Cauchy net $(x_\alpha)$ in $A$. Then there exist two $u$-uniform Cauchy nets $(y_\alpha)$ and $(z_\alpha)$ with $x_\alpha=y_\alpha+z_\alpha\cdot w$ in $Y_+$ because of $y_\alpha\leq x_\alpha$ and $z_\alpha\leq x_\alpha$. So, there are $y, \ z\in Y$ such that $y_\alpha\xrightarrow{u} y$ and $z_\alpha\xrightarrow{u}z$ because $Y$ is $u$-uniformly complete. Therefore, we get $x_\alpha=y_\alpha+z_\alpha\cdot w\xrightarrow{u}y+z\cdot w$. As a result, $A$ is also $u$-uniformly complete.
\end{proof}

\begin{thm}\label{basic theorem}
Let $(X,p,E)$ be an $LNFA$ with $X$ being $f$-subalgebra of order complete $f$-algebra $E$ and $G$ be an unital $f$-subalgebra of $X$. If $T:G\to E$ is an $E$-dominated operator and $G$ is $e$-uniform complete then there exists another $E$-dominated operator $\hat{T}:X\to E$ such that $\hat{T}(g)=T(g)$ for all $g\in G$. 
\end{thm}

\begin{proof}
First of all, if we take $T=0$ or $X=G$ then the poof is obvious. Suppose, $G$ is a proper subspace of $X$ and $T\neq 0$. So, there is a vector $w$ in $X$ so that it is not in $G$. WLOG, we assume $w\in X_+$. Then we consider the set $G_1=\{u+v\cdot w:u,v\in G\}$. Thus, by Lemma \ref{f algebra subspace}, we get that $G_1$ is also an $f$-subalgebra of $X$. Also, by using this extension, we can arrive at $X$ because $G$ is $f$-subalgebra with the multiplicative unit.
	
The extension of one step is not similar to the other Hahn-Banach theorems. It can be observed that $v\cdot w$ can be in $G$ for some $v\in G$. Thus, we have that the representation $G_1$ may not be unique. So, it causes to difficulties getting an extension of one step. Whenever it is done, by using Zorn's lemma and applying Proposition \ref{f algebra order complete}, we can get the extension of $\hat{T}$ to $X$.
	
Now, consider elements $u,v \in G$. Since $T$ is an $E$-dominated operator. Then we have
$$
T(u)+T(v)=T(u+v)\leq p(u-w+w+v)\leq p(u-w)+p(w+v)
$$
Hence, we get $T(u)-p(u-w)\leq p(w+v)-T(v)$. From there, by applying order completeness of $E$, the both 
$$
s=\sup\{T(u)-p(u-w):u\in G\}
$$
and
$$
r=\inf\{p(v+w)-T(v):v\in G\}
$$
exist in $E$. So, it is also clear $s\leq r$. Next, let's take any element $z\in E$ such that $s\leq z\leq r$ (for example we can take $z=s$).
Now, we define a map 
\begin{align*}
\hat{T}:G_1&\to E \\(u+v\cdot w)&\to\hat{T}(u+g\cdot w)=T(u)+v\cdot z.
\end{align*}
We need to show that $\hat{T}$ is a well defined operator. To prove that, we firstly prove the $E$-dominatedness fo $\hat{T}$. Let's apply $e$-uniformly completeness of $G$. Then we have that $(v+e)^{-1}$ exits for any positive element $v\in G_+$; see \cite[Theorem 146.3.]{Za}. Next, by using \cite[Theorem 11.1.]{P}, the inverse element $(v+\frac{1}{n}e)^{-1}$ exists in $G_+$ for all $n\in\mathbb{N}_+$. Then, for each $u\in G_+$ and $n\in\mathbb{N}$, we have
	$$
	z\leq r\leq p(u\cdot(v+\frac{1}{n}e)^{-1}+w)-T(u\cdot (v+\frac{1}{n}e)^{-1})
	$$
	and so, by using the $f$-subalgebraic-linear property of $p$, we get
	$$
	T(u)+(v+\frac{1}{n}e)\cdot z\leq p(u+w\cdot(v+\frac{1}{n}e))\leq p(u+w\cdot v)+\frac{1}{n}p(w).
	$$
	Thus, we have $\hat{T}(u+v\cdot w)=T(u)+v\cdot z\leq p(u+v\cdot w)$ for any $u,v\in G_+$ because $F$ is an Archimedean vector lattice. Thus, $\hat{T}$ is $E$-dominated for arbitrary $u,v\in G_+$. Now, we show for arbitrary $v\in G$. We can write $v=v^+-v^-$. By using the first observation, we can write
	\begin{equation}
	\hat{T}(u+v^+\cdot w)=T(u)+v^+\cdot z\leq p(u+v^+\cdot w)
	\end{equation}
	For the band $B_{v^+}$ generated by $v^+$, we consider the band projection $q:G\to B_{v^+}$. Then $q$ holds $q(v)=v^+$ and $q=q^2$, and it is an positive orthomorphism on $G$ because every order projection is a positive orthomorphism on vector lattices. By using \cite[Theorem 141.1.]{Za}, we can choose a positive element $t\in G_+$ such that $q(x)=x\cdot t$ for all $x\in G$. Thus we have a positive vector $t\in G_+$ so that $v^+=q(v)=v\cdot t$, and $t=e\cdot t=q(e)=q(q(e))=t^2$, and $v^+=q(v^+)=v^+\cdot t$, and $0=q(v^-)=v^-\cdot t$. Also, the equality $v^+=q(v)=v\cdot t$ implies $v^-+v=v^+=v\cdot t$, and so we vet $v^-=v\cdot(t-e)$. Thus, we obtain the following both equalities
	\begin{equation}
	t\cdot(v^+\cdot z)=(t\cdot v^+)\cdot z=v^+\cdot z
	\end{equation}
	and
	\begin{equation}
	t\cdot(v^+\cdot w)=t\cdot v^+\cdot w=t\cdot(v\cdot t)\cdot w=t^2\cdot v\cdot w=t\cdot v\cdot w.
	\end{equation}
	It follows from $(1),\ (2)$ and $(3)$ and the $f$-subalgebraic-linear property of $p$ that
	\begin{eqnarray}
	t\cdot\big(T(u)+v^+\cdot z\big)\leq t\cdot p(u+v^+\cdot w)=p(t\cdot u+t\cdot v^+\cdot w)\big]=t\cdot p(u+v\cdot w).
	\end{eqnarray}
	As one repeat the same way and use $r\leq z$, it can be seen the following inequality
	\begin{equation}
	(e-t)\cdot\big(T(u)-v^-\cdot z\big)\leq (e-t)\cdot p(u+v\cdot w).
	\end{equation}
	Therefore, by summing up the inequalities $(4)$ and $(5)$, we can get the following result
	\begin{equation}
	T(u)+v\cdot z\leq p(u+v\cdot w)
	\end{equation}
	for arbitrary $v\in G$ and $u\in G_+$. Lastly, one can show for arbitrary element $u\in G$. Therefore, we get that $\hat{T}$ is $E$-dominated. Now, we show  well defined of $\hat{T}$. Let's take arbitrary elements $u_1,\ u_2,\ v_1,\ v_2\in G$ such that $u_1+v_1\cdot w=u_2+v_2\cdot w$. It follows from $(6)$ that $T(u_1-u_2)+(v_1-v_2)\cdot z\leq p\big((u_1-u_2)+(v_1-v_2)\cdot w)\big)=p(0)=0$ and $T(u_2-u_1)+(v_2-v_1)\cdot z\leq p\big((u_2-u_1)+(v_2-v_1)\cdot w)\big)=p(0)=0$. As a result, we get $\hat{T}(v_1+g_1\cdot w)=\hat{T}(v_2+g_2\cdot w)$. Therefore, we have obtained that the map $\hat{T}$ is well defined. On the other hand, by using the linearity of $T$, one can show that $\hat{T}$ is a linear map (or, operator) from $G_1$ to $F$. Expressly, $\hat{T}$ is $E$-dominated operator by $f$-subalgebraic-linear map $p$. By applying Zorn's lemma under the desired conditions, we provide the extension of $\hat{T}$ to all of $X$.
\end{proof}

Under the condition of Theorem \ref{basic theorem}, we have the following results.
\begin{cor}
	If $(X,p,E)$ is a decomposable $LNFA$ then we have $[\hat{T}]=[T]$. 
\end{cor}

\begin{proof}
	Since $T$ is $E$-dominated operator, it is dominated. Indeed, Since $\lvert T(g)\rvert\leq p(g)$, we have $p(T(g))\leq p(p(g))$ (for example we can take a dominant $S=p$). Also, it follows from \cite[Theorem 4.1.2.]{Ku} that $T$ has the exact dominant $[T]$. Now, consider the $f$-subalgebra $G_1$ of $X$ in the proof of Theorem \ref{basic theorem}. For $v=0$ the addition unit and $u\in G$, we have
	$$
	\hat{T}(u)=T(u)\leq \lvert T(u)\rvert\leq S(p(u))
	$$
	and also
	$$
	-\hat{T}(u)=-T(u)\leq \lvert T(u)\rvert\leq S(p(u)).
	$$
	Therefore, we get $\lvert \hat{T}(u)\rvert\leq S(p(u))$ for each $u\in G$. Hence, $\hat{T}$ is also dominated by $S$, and so, we get $[\hat{T}]\leq [T]$. On the other hand, by considering the $maj(T)$ and $maj(\hat{T})$, we have $[T]\leq [\hat{T}]$. As a result, we get the desired result.
\end{proof}

\begin{cor}
	Let $Y$ be an unital and $e$-uniform complete $p$-closed $f$-subalgebra of $X$. If every non zero positive element has inverse in $Y$ then, for each $y_0\notin Y$, we have a map $F:X\to E$ such that $F(Y)=0$ and $F(y_0)>0$.
\end{cor}

\begin{proof}
	Let's take a set $Y_1=\{u+v\cdot y_0:u,v\in Y\}$ and $w=\inf\{p(y+ y_0):y\in Y\}\geq0$. Then we show $w\neq0$. Assume it is not hold true, i.e., $w=0$. For any $a_1,\ a_2\in A$, it is enough to show that $a_1\wedge a_2\in A$. For proving that, we consider \cite[Theorem 2.1.2]{Ku} and take a band $B=(a_1-a_\vee a_2)$. The there is a band projection $\pi_B:E\to B$. Then we have another projection $\pi'_B$ on $X$ such that $\pi_B\big(p(x)\big)=p\big(\pi'_B(x)\big)$. So, we have
	\begin{eqnarray*}
		\pi_B(a_1)+\pi^d_B(a_1)&=&\pi_B(a_1\vee a_2+a_1\wedge a_2-a_2)+\pi^d_B(a_1\vee a_2+a_1\wedge a_2-a_2)\\&=& \pi_B(a_1\wedge a_2)+\pi^d_B(a_1\wedge a_2)\\&=&a_1\wedge a_2.
	\end{eqnarray*}
	Now, take $y_1,y_2\in Y_+$ so that $a_1=p(y_1+y_0)$ and $a_2=p(y_2+y_0)$. Thus, we can get
	\begin{eqnarray*}
		a_1\wedge a_2=\pi_B(a_1)+\pi^d_B(a_1)&=&\pi_B\big(p(y_1+y_0)\big)+\pi^d_B\big(p(y_2+y_0)\big)\\&=&p\big(\pi'_B(y_1+y_0)\big)+p\big(\pi'^d_B(y_2+y_0)\big)\\&=& p\big(\pi'_B(y_1+y_0)+\pi'^d_B(y_2+y_0)\big)\\&=& p\big(\pi'_B(y_1+y_0)+\pi'^d_B(y_2+y_0)\big)\\&=& p\big(y_0+\pi'_B(y_1)+\pi'^d_B(y_2)\big)
	\end{eqnarray*}
	Therefore, we can see $a_1\wedge a_2\in A$. Thus, one can see $a_1\wedge a_2\leq a_1$ and $a_1\wedge a_2\leq a_2$. So, $A$ is downward directed set. Therefore, we can take $A$ as a net in $E$. Since $p(y_\alpha-y_0)=p(y_0-y_\alpha)\downarrow 0$, we have $y_\alpha\xrightarrow{p}y_0$. Thus, we get $y_0\in Y$ because $Y$ is $p$-closed set. Which is contradict with $y_0\notin Y$, and so, we have $w>0$. Next, we define a map $T:Y_1\to E$ by $f(u+v\cdot y_0)=v\cdot w$. Then $T$ is linear and $T(Y)=0$. Moreover, $T$ is also $E$-dominated. Indeed, we can write $p(u+v\cdot y_0)=v\cdot p(v^{-1}\cdot u+y_0)\geq v\cdot w=T(u+v\cdot y_0)$. It follows from the Theorem \ref{basic theorem} that there exists a map from $X$ to $E$ satisfying the desired result.
\end{proof}

For the next result, we consider the $f$-algebraic spaces $\mathcal{L}(E)\subseteq Orth(E)\subseteq Orth^\infty(E)$ in Example \ref{example of orh}.
\begin{cor}
Let $E$ be an order complete vector lattice. $\big(Orth(E),\lvert\cdot\rvert,Orth^\infty(E)\big)$ is an $LNFA$. Moreover, If $T:\mathcal{L}(E)\to Orth^\infty(E)$ an $E$-dominated operator then it has an extension to $Orth(E)$.
\end{cor}

\begin{proof}
Since $E$ be an order complete vector lattice, we see that $Orth^\infty(E)$ is order complete $f$-algebra; see \cite[p.14]{BGKKKM}. Moreover, we can say that $\big(Orth(E),\lvert\cdot\rvert,Orth^\infty(E)\big)$ is an $LNFA$ because $Orth(E)$ is $f$-subalgebra of $Orth^\infty(E)$ and $\lvert\cdot\rvert$ has the $f$-subalgebraic-linear property.
	
By applying \cite[Theorem 3.1.]{WE}, we can see that $L(E)$ is order complete because $E$ is order complete. Moreover, by using \cite[Theorem 42.6.]{LZ}, we also get that $L(E)$ is $e$-uniform complete because $L(E)$ has unit $I_E$. Then, we have an $E$-dominated extension $T$ to $(Orth(E)$.
\end{proof}

\end{document}